\documentclass[11pt]{article}
\usepackage{amssymb}
\usepackage[all,arc]{xy}
\usepackage{mathrsfs}
\usepackage{extarrows}
\usepackage{xcolor}
\usepackage{amsmath}
\usepackage{amsfonts,epsfig}
\usepackage{amssymb,bm}
\usepackage{graphicx}
\usepackage{subfig}
\usepackage{pgf,pgfarrows,pgfnodes,pgfautomata,pgfheaps}
\usepackage{appendix}

\usepackage{tabularx}  
\usepackage{booktabs}
\usepackage{multirow}
\usepackage{array,makecell}

\usepackage{algorithm}
\usepackage{algpseudocode}
\usepackage{pifont}

\usepackage{tikz}

\numberwithin{equation}{section}

\usepackage{amsmath,amsthm,amsfonts,amssymb,amscd,color}
\usepackage{hyperref}
\hypersetup{hypertex=true,
colorlinks=true,
linkcolor=blue,
anchorcolor=blue,
citecolor=blue}

\allowdisplaybreaks[4]
\newtheorem{theo}{Theorem}[section]
\newtheorem{lem}[theo]{Lemma}

\begin{document}

\newcommand{\E}{\mathbb{EX}}
\newcommand{\EX}{{\rm EX}}

\title{The Tur\'an number of path-star forests}
\author{Xiaona Fang$^{a,b}$, Yaojun Chen$^a$, Lihua You$^{b,}$\thanks{Corresponding author: ylhua@scnu.edu.cn}\\
{\small $^a$School of Mathematics, Nanjing University,}\\ {\small Nanjing, 210093, P. R. China}\\
{\small $^b$School of Mathematical Sciences, South China Normal University,}\\ {\small Guangzhou, 510631, P. R. China}
}
\date{}
\maketitle

\begin{abstract}
The Tur\'an number of a graph $H$, denoted by $ex(n,H)$, is the maximum number of edges in any graph on $n$ vertices containing no $H$ as a subgraph. A linear (star) forest is a forest consisting of paths (stars). A path-star forest $F$ is a forest consisting of paths and stars. In this paper, we determine $ex(n,F)$ for sufficiently large $n$ and characterize the corresponding extremal graphs, and our result generalizes previous known results on the Tur\'an numbers of linear forests and star forests.

\end{abstract}

\section{Introduction}\label{sec1}
\hspace{1.5em}In this paper, we only consider simple graphs. Let $K_n$, $P_n$, $S_{n-1}$ be a complete graph, a path and a star on $n$ vertices, respectively. For a graph $G$, we use $V(G)$, $|G|$, $E(G)$, $e(G)$ to denote the vertex set, the number of vertices, the edge set and the number of edges of $G$, respectively. For $S,S'\subseteq V(G)$ and $S\cap S'=\emptyset$, denote by $e(S,S')$ the number of edges between $S$ and $S'$. For a vertex $v \in V(G)$, let $N_G(v)$ denote the set of vertices in $G$ which are adjacent to $v$ and  $d_G(v) = |N_G(v)|$.

Let $G$ and $H$ be two disjoint graphs, denote by $G \cup H$ the disjoint union of $G$ and $H$ and by $k  G$ the disjoint union of $k$ copies of a graph $G$. Denote by $G\vee H$ the graph obtained from $G \cup H$ by adding edges between all vertices of $G$ and all vertices of $H$. We use $\overline{G}$ to denote the complement of the graph $G$. For any set $S \subseteq V (G)$, let $G[S]$ denote the subgraph of $G$ induced by $S$, and let $|S|$ denote the cardinality
of $S$. For a graph $G$ and its subgraph $H$, let $G - H$ denote the subgraph induced by
$V (G)\setminus V (H)$. If $H$ consists of a single vertex $x$, then we simply write $G - x$.

Let $M_k=\frac{k}{2}K_2$ or $\frac{k-1}{2}K_2\cup K_1$ depending on whether $k$ is even or odd, $G_1(n,k)=K_k \vee \overline{K}_{n-k}$, $G_1^+(n,k)=K_k \vee (K_2 \cup \overline{K}_{n-k-2})$ and $G_2(n,k,\ell)=K_k \vee (d K_{\ell-1} \cup K_r)$, where $n=k+d(\ell-1)+r$, $0 \leq  r <\ell-1$ (see Figure \ref{fig1}).

\begin{figure}
	\centering
	\begin{tikzpicture}
		\node[anchor=south west,inner sep=0] (image) at (0,0) {\includegraphics[width=1\textwidth]{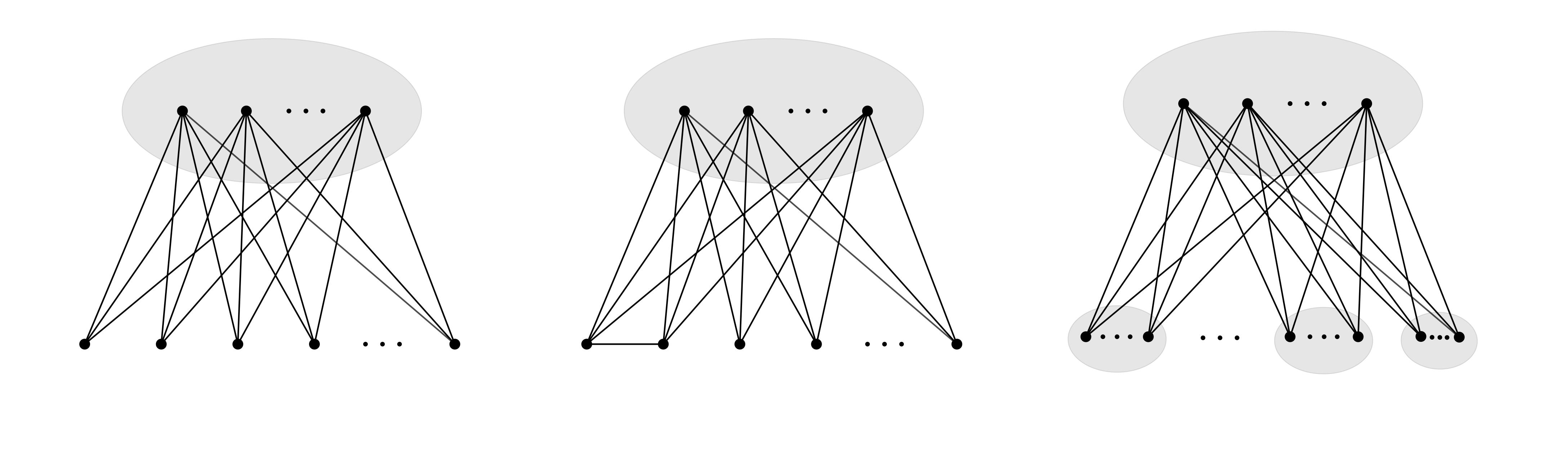}};
		\begin{scope}[
			x={(image.south east)},
			y={(image.north west)}
			]
			\node [black, font=\bfseries] at (0.175,0.84) {$K_k$};
			\node [black, font=\bfseries] at (0.5,0.84) {$K_k$};
			\node [black, font=\bfseries] at (0.82,0.85) {$K_k$};
			\node [black, font=\bfseries] at (0.72,0.19) {$K_{\ell-1}$};
			\node [black, font=\bfseries] at (0.85,0.19) {$K_{\ell-1}$};
			\node [black, font=\bfseries] at (0.92,0.19) {$K_r$};
			\node [black, font=\bfseries] at (0.175,0.05) {$G_1(n,k)$};
			\node [black, font=\bfseries] at (0.5,0.05) {$G_1^+(n,k)$};
			\node [black, font=\bfseries] at (0.81,0.05) {$G_2(n,k,\ell)$};
		\end{scope}
	\end{tikzpicture}
	\caption{Three graphs.}
	\label{fig1}
\end{figure}
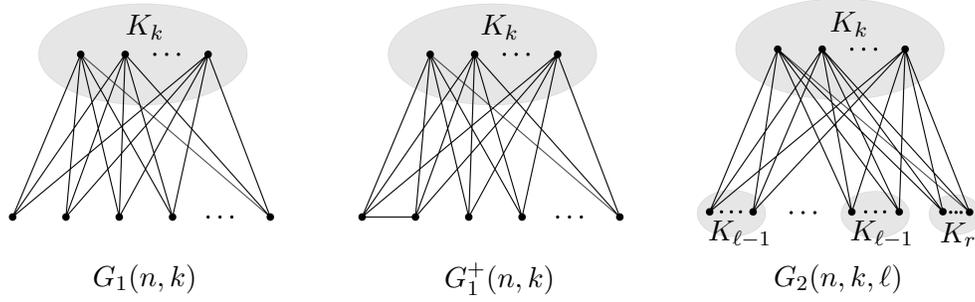

Let $n \geq m\geq \ell\geq 2$ be three positive integers and $n=(m-1)+d(\ell-1)+r$ with $d\geq 0$ and $0 \leq  r <\ell-1$. Define 
$$[n,m,\ell]=\binom{m-1}{2}+d\binom{\ell-1}{2}+\binom{r}{2}.$$

For convenience, let $[p]=\{1,2,\dots,p\}$, where $p$ is a positive integer.

A graph is $H$-free if it does not contain a copy of $H$ as a subgraph. The Tur\'an number of a graph $H$, denoted by $ex(n, H)$, is the maximum number of edges in any $H$-free graph on $n$ vertices. Let $\mathbb{EX}(n, H)$ denote the set of $H$-free graphs on $n$ vertices with $ex(n, H)$ edges. If $\mathbb{EX}(n, H)$ contains only one
graph, we may simply use \EX$(n, H)$ instead. A graph in $\mathbb{EX}(n, H)$ is called an extremal graph for $H$. 

The problem of determining Tur\'an number of graphs traces its history back to 1907, when Mantel \cite{mantel} proved $ex(n, C_3)=\lfloor \frac{n^2}{4} \rfloor$.  In 1941, Tur\'an \cite{turan1,turan2} determined $ex(n,K_r)$ and the corresponding extremal graphs, which extends Mantel’s theorem. The study of Tur\'an numbers of forests began with the famous result of Erd\H os and Gallai \cite{erdos1959}, which determined the Tur\'an numbers of matchings and paths. Faudree and Schelp \cite{faudree1975} further characterized  the
extremal graphs for $P_{\ell}$ and obtained the following.

\begin{theo}{\rm (\cite{faudree1975})}\label{p1}
	Let $n=d(\ell-1)+r$, where $d \geq 1$ and $0\leq  r<\ell-1$. Then
	$ex(n, P_{\ell}) = [n,\ell,\ell]$.
	Furthermore,
	$$\mathbb{EX}(n,P_{\ell})=\left\{ d K_{\ell-1} \cup K_r, (d-s-1) K_{\ell-1}\cup 
\left(K_{\frac{\ell-2}{2}} \vee
\overline{K}_{\frac{\ell}{2}+s(\ell-1)+r}
\right), 0\leq s \leq d-1 \right\}$$ 
 if $\ell$ is even and $r = \frac{\ell}{2}$ or $\frac{\ell-2}{2}$, and 
	\EX$(n, P_{\ell}) = d K_{\ell-1} \cup K_r$ otherwise.
\end{theo}

A path-star forest is a forest whose connected components are paths and stars. Let $F=\mathop{\cup}\limits _{i=1}^{p} P_{\ell _i} \mathop{\cup}\limits _{j=1}^{q} S_{a _j}$ be a path-star forest. If $q=0$, then $F$ is called a linear forest, and if $p=0$, then $F$ is called a star forest.  For convenience, we set
\begin{equation*}
	\delta_F=\sum\limits_{i=1}^{p} \left\lfloor\frac{\ell_i}{2}\right\rfloor.
\end{equation*}

For linear forests, Lidick\'y et al. \cite{lidicky2013} determined the value of $ex(n, \cup_{i=1}^p P_{\ell _i})$ for sufficiently large $n$.

\begin{theo}{\rm (\cite{lidicky2013})}\label{pp}
	Let $F=\cup_{i=1}^p P_{\ell _i}$. If at least
	one $\ell_i$ is not $3$, then for sufficiently large $n$,
	$$ex(n,F)=(\delta_F-1)\left(n-\frac{\delta_F}{2}\right)+c,$$
	where $c = 1$ if all $\ell_i$ are odd and $c = 0$ otherwise. Moreover, \EX$(n,F)=G_1^+(n, \delta_F-1) $ if all $\ell_i$ are odd, and \EX$(n,F)=G_1(n, \delta_F-1)$ otherwise.
\end{theo}

Zhu and Chen \cite{zhu2022} improved the lower bound of $n$ in Theorem \ref{pp} for $F=\cup_{i=1}^p P_{\ell _i}$ with each $\ell_i\not=3$.
\begin{theo}{\rm (\cite{zhu2022})}\label{kPnZ}
	Let $p\geq 2$, $F=\cup_{i=1}^p P_{\ell _i}$ with $\ell_i \neq 3$ for $i\in [p]$ and $s=\sum_{i=1}^p \ell_i$.
	If $n\geq \frac{5(s-1)(s-2)^2}{4(\delta_F-1)}+\delta_F-1$, then 
	$$ex(n,F)=(\delta_F-1)\left(n-\frac{\delta_F}{2}\right)+c,$$
	where $c = 1$ if all $\ell_i$ are odd and $c = 0$ otherwise. Moreover,  \EX$(n,F)=G_1^+(n, \delta_F-1) $ if all $\ell_i$ are odd, and \EX$(n,F)=G_1(n, \delta_F-1)$ otherwise.
\end{theo}

Yuan and Zhang \cite{yuan2017,yuan2021} completely determined the value of $ex(n, pP_3)$ and characterized
all the extremal graphs for all $n$. Furthermore, they determined the Tur\'an numbers of the linear forests containing at most
one odd path for all $n$. 

\begin{theo}{\rm (\cite{yuan2017})}\label{p3}
	For $n> 5p-1$, $ex(n,p P_3)= (p-1)(n-\frac{p}{2}) + \lfloor\frac{n-p+1}{2}\rfloor $. Moreover, \EX $ (n, p P_3)=K_{p-1}\vee M_{n-p+1}$.
\end{theo}

\begin{theo}{\rm (\cite{yuan2021})}\label{P2n}
	Let $F=\cup_{i=1}^p P_{\ell _i}$ with $\ell_1 \geq \cdots \geq \ell_p\geq 2$, $n\geq \sum_{i=1}^p \ell_i$ and $s_j=\sum_{i=1}^j \ell_i$. If  at most one of $\ell_1,\ell_2,\dots,\ell_p$ is odd, then 
	$$ex(n,F)=\max \left\{[n,s_1, \ell_1], [n,s_2, \ell_2], \ldots, \left[n,s_p, \ell_p \right], (\delta_F-1)\left(n-\frac{\delta_F}{2}\right)  \right\}. $$
	Moreover, if all of $\ell_1,\dots,\ell_p$ are even, then each extremal graph is isomorphic to one of the  graphs as follows: 
	$\E(n,P_{\ell_1})$, $K_{s_j-1}\cup H_j$ ($j=2,...,p$),  $G_1(n,\delta_F-1)$, where $H_j \in$ $\E(n-s_j+1,P_{\ell_j})$ for $j\in [p]$.
\end{theo}


For  star forests, Lidick\'y et al. \cite{lidicky2013} determined $ex(n, \cup_{j=1}^q S_{a _j})$ for sufficiently large $n$.


\begin{theo}{\rm (\cite{lidicky2013})}\label{Sncup} Let $F=\cup_{j=1}^q S_{a_j}$with $a_1\geq \cdots\geq a_q$. Then for sufficiently large $n$, 
$$ex(n, F) =\max\limits_{1\leq j \leq q} \left\{(j-1)\left(n-\frac{j}{2}\right)+\left \lfloor \frac{a_j-1}{2}(n-j+1) \right \rfloor \right\}$$
 and the extremal graph is $K_{j-1}\vee$\EX$(n-j+1, S_{a_j})$ for some $j\in [q]$.
\end{theo}


Lan et al. \cite{lan2019} improved the lower bound of $n$ in Theorem \ref{Sncup} for $q S_{a}$.

\begin{theo}{\rm (\cite{lan2019})}\label{lanS}	
	If $a \geq 3$ and $n \geq q(a^2+a+1) -\frac{a}{2} (a-3)$, then
	$ex(n, q S_{a}) =(q-1)(n-\frac{q}{2})+\left \lfloor \frac{a-1}{2}(n-q+1) \right \rfloor .$
	Moreover, the extremal graph is $K_{q-1} \vee \EX(n-q + 1,S_{a})$.
\end{theo}

Li et al. \cite{li2022} determined the Tur\'an number of $q S_{a}$ for all $n$, where $q \geq 2$ and $a \geq 3$. In fact, when $q=1$ in Theorem \ref{lanS}, we have $ex(n,S_a)= \lfloor \frac{(a-1)n}{2} \rfloor $. The extremal graphs are almost $(a-1)$-regular graphs, which are $(a-1)$-regular if $(a-1)n$ is even, and if $(a-1)n$ is odd then every vertex of each extremal graph has degree $a-1$ except one vertex of degree $a-2$.

%

\vskip 2mm
Beyond the results on the Tur\'an numbers of linear forests and star forests,  Lidick\'y et al. \cite{lidicky2013} started to consider the Tur\'an number of path-star forests. They determined $ex(n, pP_4\cup q S_3)$ for sufficiently large $n$.   Lan et al. \cite{lan2019} proved that the extremal graph for  $pP_4\cup q S_3$ is unique when $n$ is large enough. Zhang and Wang \cite{zhang2022} determined $ex(n, p P_6\cup q S_5)$ and characterized all the extremal graphs. Recently, Fang and Yuan \cite{fang2023} extended all of the above results to a more general situation as below.

\begin{theo}{\rm (\cite{fang2023})}\label{1,k}
	Suppose $n=q+d(\ell-1)+r\geq (\ell^2-\ell+1)q+\frac{\ell^2+3\ell-2}{2}$, where $\ell\geq 4$, $0\leq r<\ell-1$. Then $ex(n, P_{\ell} \cup q S_{\ell-1}) = e(G_2(n,q,\ell)) $. Moreover, 
$$\E (n, P_{\ell} \cup q S_{\ell-1} )=
\begin{cases}
	\left\{G_2(n,q,\ell), G_1\left(n, \frac{\ell}{2}+q-1\right)  \right\}, & \text{if } \ell \text{ is even, } r\in \left\{\frac{\ell}{2}, \frac{\ell-2}{2}\right\}, \\
	\left\{ G_2(n,q,\ell) \right\}, & \text{otherwise}.
\end{cases}$$
\end{theo}

\begin{theo}{\rm (\cite{fang2023})}\label{k1k2}
	Suppose $n\geq (2\ell^2+3\ell-4)p+(4\ell^2-2\ell+1)q+3$, where $p\geq 2$, $\ell\geq 2$. Then $ex(n, p P_{2\ell} \cup q S_{2\ell-1}) =(\ell p+ q-1)(n-\frac{\ell p+ q}{2}) $. Moreover, \EX$(n, p P_{2\ell} \cup q S_{2\ell-1})=G_1(n, \ell p+ q-1)$.
\end{theo}

One can see that the two theorems above only deal with the path-star forests in which all paths and stars have the same order. In this paper, we try to determine the Tur\'an number of any path-star forest $F=\mathop{\cup}\limits _{i=1}^{p} P_{\ell _i} \mathop{\cup}\limits _{j=1}^{q} S_{a _j}$ with $p+q \geq 1$, and the corresponding extremal graphs. For convenience in stating our main result, we introduce a parameter with respect to $F$ as below. 
 Recall that $\delta_F=\sum\limits_{i=1}^{p} \left\lfloor\frac{\ell_i}{2}\right\rfloor$. Define 
 $$\beta_F= q+\delta_F-\mu_F,$$ 
where $\mu_F=1$  if $p=1$ and  $\ell_1$  is even or  $p\geq 2$ and $\exists \ \ell_i\neq 3$, and $\mu_F=\frac{1}{2}$ otherwise.
\vskip 3mm

The main result of this paper is the following.
\begin{theo}\label{main}
	Let $F=\mathop{\cup}\limits _{i=1}^{p} P_{\ell _i} \mathop{\cup}\limits _{j=1}^{q} S_{a _j}$ with $p+q\geq 1$ and  $a_1\geq \cdots \geq  a_{q}\geq 3$.
	Then for sufficiently large $n$,
	\begin{equation}\label{eq1main}
		ex(n,F)=\binom{q}{2}+q(n-q)+ex\left(n-q, \mathop{\cup}\limits _{i=1}^{p} P_{\ell _i}\right)
	\end{equation}
	if $q=0$, or $\beta_F> \max\limits _{1 \leq j\leq q} \big\{j-1+\frac{a_j-1}{2}\big\}$ and $a_q\geq \ell_1-1$ in the case when $p=1$, and  $ex(n,F)=ex\big(n, \cup_{j=1}^{q} S_{a _j} \big)$  if $\beta_F\leq  \max\limits _{1 \leq j\leq q} \big\{j-1+\frac{a_j-1}{2}\big\}$. 
	
	Moreover, in the former case, $\EX(n,F)=K_{q}\vee \EX\big(n-q, \cup_{i=1}^{p} P_{\ell _i} \big)$ if  $p\geq 2$ and
	$$\E(n,F)=\begin{cases}
		\big\{G_2\big(n,q,\ell_1\big) \big\}\cup \mathbb{G},	& \text{ if } \ell_1 \text{ is even and } r \in \big\{\frac{\ell_1}{2},\frac{\ell_1-2}{2} \big\}, \\[2pt]
		\big\{G_2\big(n,q,\ell_1\big)\big \}, & \text{ otherwise, }  
	\end{cases}$$
   when $p=1$ and  $a_{q}\geq \ell_1-1$, where $\,n=q+d(\ell_1 -1)+r$, $0\leq r < \ell_1-1$, 
	$$\mathbb{G}=\begin{cases}
		\{G(d-1), G(0) \}, & \mbox{if } a_{q}=\ell_1=4,\ r=1, \\[1pt]
		\{G(d-1)\}, & \mbox{otherwise,}
	\end{cases}$$
	and $G(s)= K_{q} \vee \left( (d-s-1) K_{\ell_1-1}\cup 
	\left(K_{\frac{\ell_1-2}{2}} \vee
	\overline{K}_{\frac{\ell_1}{2}+s(\ell_1-1)+r}
	\right) \right)$; in the latter case,
	$\EX(n,F)=\EX\big(n, \cup_{j=1}^{q} S_{a _j} \big)$.
\end{theo}

The proof of Theorem \ref{main} will be presented in Section \ref{PS}. In Section \ref{remark}, we improve the range of $n$ in Theorem \ref{main} for $\cup _{i=1}^{p} P_{\ell _i} \cup q S_a$, which contains Theorems \ref{1,k} and  \ref{k1k2} as special cases.

\section{Proof of Theorem \ref{main}}\label{PS} 

\hspace{1.5em} If $q=0$, then $F=\cup _{i=1}^{p} P_{\ell _i}$, and so the result follows from Theorems \ref{p1}, \ref{pp} and \ref{p3}. Now, we assume that $q\geq 1$ and show Theorem \ref{main} in two separate parts.

\vskip 3mm
\noindent \textbf{Part I.} $\beta_F> \max\limits _{1 \leq j\leq q} \big\{j-1+\frac{a_j-1}{2}\big\}$ and $a_q\geq \ell_1-1$  when $p=1$.
\vskip 2mm

It is not difficult to check the sizes of the graphs described in Theorem \ref{main} are the expected value of
$ex(n, F)$. So to obtain the lower bound for $ex(n, F)$, it suffices to show that $K_{q}\vee \EX\big(n-q,\cup_{i=1}^{p} P_{\ell _i} \big)$ for  $p\geq 2$ and $G_2\big(n,q,\ell_1\big)$ are $F$‐free.

If $p=1$ and $G_2(n,q,\ell_1)=K_{q} \vee (d K_{\ell_1-1} \cup K_r)$ contains $F$ as a subgraph, then for any $j\in[q]$, $S_{a_j}$  contains at least one vertex of $K_{q}$ since $a_{q}\geq \ell_1-1$, and $P_{\ell_1}$ contains at least
one vertex of $K_{q}$, which implies $K_{q}$ has $q+1$ vertices,  a contradiction. Hence, $G_2(n,q,\ell_1)$ is $F$-free.

If $p\geq 2$ and at least
one $\ell_i$ is not $3$, then $K_{q}\vee \EX\big(n-q, \cup_{i=1}^{p} P_{\ell _i} \big)=G_1^+(n, q+\delta_F-1)$ or $G_1(n,q+\delta_F-1)$ by Theorem \ref{pp}.  Since each $S_{a_j}$ contains at least one vertex of $K_{q+\delta_F-1}$ and each $P_{\ell_i}$ contains at least $\lfloor\frac{\ell_i}{2}\rfloor$ vertices of $K_{q+\delta_F-1}$ in $G_1^+(n, q+\delta_F-1)$ or $G_1(n,q+\delta_F-1)$, this requires $K_{q+\delta_F-1}$ having $q+\delta_F$ 
vertices to form an $F$, which is impossible. Hence, $K_{q}\vee \EX\big(n-q, \cup_{i=1}^{p} P_{\ell _i} \big)$ is $F$-free in this case.

If $p\geq 2$  and  $\ell_1=\cdots=\ell_{p}=3$, then $K_{q}\vee \EX\big(n-q, \cup_{i=1}^{p} P_{\ell _i} \big)=K_{p+q-1}\vee M_{n-p-q+1}$ by Theorem \ref{p3}. If $K_{p+q-1}\vee M_{n-p-q+1}$ contains $F$ as a subgraph, then each $S_{a_j}$ and each $P_3$ contains at least
one vertex of $K_{p+q-1}$, which requires  $K_{p+q-1}$ having $p+q$ vertices,  a
contradiction. Hence, $K_{q}\vee \EX\big(n-q, \cup_{i=1}^{p} P_{\ell _i} \big)$ is also $F$-free in this case.
 
Therefore, we have
\begin{equation}\label{geq}
	ex(n,F)\geq \binom{q}{2}+q(n-q)+ex\left(n-q, \mathop{\cup}\limits _{i=1}^{p} P_{\ell _i}\right).
\end{equation}

\vskip 2mm
Now, we proceed to prove Part I of Theorem \ref{main} by induction on $q$. 

By Theorem \ref{p1}, Theorem \ref{pp} and Theorem \ref{p3}, we have 
\begin{equation}\label{case1}
	\binom{q}{2}+q(n-q)+ex\left(n-q, \mathop{\cup}\limits _{i=1}^{p} P_{\ell _i}\right)=\beta_F\cdot n- \gamma_F,
\end{equation}
where 
\begin{equation}\label{case11}
	\gamma_F=\begin{cases}
		\frac{q^2+(\ell_1-1)(q+r)-r^2}{2},	& \text{ if } p=1 \text{ and } a_{q}\geq \ell_1-1, \\[3.5pt]
		\frac{(p+q)^2-1}{2}+c_1,	&  \text{ if } p\geq 2 \text{ and } \ell_1=\cdots=\ell_{p}=3, \\[4pt]
		\frac{1}{2} (q+\delta_F) (q+\delta_F -1)-c_2, & \text{ if } p\geq 2, \exists\ \ell_i \neq 3,
	\end{cases}
\end{equation}
$c_1 = \frac{1}{2}$ if $n-p-q+1$ is odd and $c_1 = 0$ otherwise, $c_2 = 1$ if all $\ell_i$ are odd and $c_2 = 0$ otherwise.

Let $G\in \mathbb{EX}(n, F)$. Then $e(G)=ex(n,F)$. Since $\beta_F> \max\limits _{1 \leq j\leq q} \big\{j-1+\frac{a_j-1}{2}\big\}$ and $n$ is sufficiently large, by (\ref{geq}) and (\ref{case1}), we have
\begin{equation}\label{exSn}
	\begin{split}
		e(G)& \geq \binom{q}{2}+q(n-q)+ex\left(n-q, \mathop{\cup}\limits _{i=1}^{p} P_{\ell _i}\right)= \beta_F\cdot n- \gamma_F\\
		&> \max_{1\leq j \leq q} \left\{ \left\lfloor \left(j-1+\frac{a_j-1}{2}\right) n-  \frac{1}{2} (j-1)\big(j+a_j-1\big) \right\rfloor \right\}\\
		&= ex \left(n,\mathop{\cup}\limits _{j=1}^{q} S_{a _j}\right),
	\end{split}
\end{equation}
which implies that $G$ contains $\cup_{j=1}^{q} S_{a _j}$ as a subgraph by Theorem \ref{Sncup}.

We first prove (\ref{eq1main}) holds and then discuss the extremal graphs. It is clear that (\ref{eq1main}) holds for $q=0$. Assume that $q=1$. Since $G$ contains $S_{a_1}$ as a subgraph, $G-S_{a_1}$ is $\cup _{i=1}^{p} P_{\ell _i}$-free. Let $m_0$ be the number of edges incident with the vertices of $S_{a_{1}}$ in $G$. Since $n$ is sufficiently large, by (\ref{geq}), we have 
\begin{equation*}
	\begin{split}
			m_0 & =e(G)-e(G-S_{a_1})\geq n-1+ex \left(n-1, \mathop{\cup}\limits _{i=1}^{p} P_{\ell _i}\right)-ex \left(n-a_1-1,\mathop{\cup}\limits _{i=1}^{p} P_{\ell _i} \right) \\
			& \geq n-1 \geq (a_1+1)\left(\sum_{i=1}^{p} \ell_i+a_1 \right).
	\end{split} 
\end{equation*}
Thus the $S_{a_1}$ has a vertex $v$ such that $d(v)\geq \sum_{i=1}^{p} \ell_i+a_1$. It implies that $G-v$ is $\cup _{i=1}^{p} P_{\ell _i}$-free and so $e(G-v)\leq 
ex(n-1, \cup _{i=1}^{p} P_{\ell _i})$. Because $e(G)=d(v)+e(G-v)$ and $d(v)\leq n-1$, we have 
\begin{equation}\label{q=1}
	e(G)\leq n-1+ex \left(n-1, \mathop{\cup}\limits _{i=1}^{p} P_{\ell _i}\right).
\end{equation}
Combining (\ref{geq}) with (\ref{q=1}), (\ref{eq1main}) holds.

Suppose that $q\geq 2$ and Theorem \ref{main} holds for $q-1$. Since $G$ is $F$-free, $G-S_{a_{1}}$ is also $(F-S_{a_{1}})$-free. Note that $\beta_{F-S_{a_{1}}}=\beta_F-1 > \max\limits _{2 \leq j\leq q} \big\{(j-1)-1+\frac{a_j-1}{2}\big\}$. By induction hypothesis, we have
\begin{equation}\label{-Sa}
	\begin{split}
		e\left(G-S_{a_1}\right)&\leq ex\left(n-a_1-1, F-S_{a_1}\right)\\
		&\leq \binom{q-1}{2}+(q-1)(n-a_1-q)+ex\left(n-a_1-q, \mathop{\cup}\limits _{i=1}^{p} P_{\ell _i}\right).
	\end{split}
\end{equation}
Let $m_0$ be the number of edges incident with the vertices of the $S_{a_{1}}$ in $G$. Since $n$ is sufficiently large, by (\ref{geq}) and (\ref{-Sa}), we have
\begin{equation}\label{m0}
		m_0 =  e(G)- e\left(G-S_{a_1}\right) \geq n+O(1).
\end{equation}
Since $n$ is sufficiently large, the $S_{a_1}$ contains a vertex $v$ of degree at least $\sum_{i=1}^{p} \ell_i+ \sum_{j=1}^{q} a_j+q-1$ in $G$. Then $G-v$ is $(F-S_{a_1})$-free. By induction hypothesis, we have $e(G-v)\leq 
ex(n-1, F-S_{a_1}) \leq \binom{q-1}{2}+(q-1)(n-q)+ex\left(n-q, \mathop{\cup}\limits _{i=1}^{p} P_{\ell _i}\right)$. Thus
\begin{equation}\label{q>1}
	\begin{split}
		e(G)&=d(v)+e(G-v)
		\leq n-1+ex(n-1, F-S_{a_1})\\ 
		&\leq \binom{q}{2}+q(n-q)+ex\left(n-q, \mathop{\cup}\limits _{i=1}^{p} P_{\ell _i}\right). 
	\end{split}	
\end{equation}
Combining (\ref{geq}) with (\ref{q>1}), (\ref{eq1main}) holds.

\vskip 2mm
Finally, we characterize the extremal graphs for $F$. 

When $q=1$, the equality holds in (\ref{q=1}) if and only if $G=K_1 \vee \EX (n-1, \cup _{i=1}^{p} P_{\ell _i})$ is $F$-free. If $p\geq 2$, noting that $G_1^+(n,q+\delta_F-1)$ or $G_1(n,q+\delta_F-1)$ or $K_{p+q-1}\vee M_{n-p-q+1}$ is $F$-free, as shown before, by  Theorems \ref{pp} and \ref{p3}, we have $G=K_{1}\vee \EX(n-1, \cup _{i=1}^{p} P_{\ell _i} )$.

If $p=1$, then $\EX (n-1, P_{\ell_1})=d K_{\ell_1-1} \cup K_r$ or $\EX (n-1, P_{\ell_1})\in \{(d-s-1) K_{\ell_1-1}\cup H~|~0\leq s\leq d-1\}$ by Theorem \ref{p1}, where $n-1=d(\ell_1-1)+r$, $0\leq  r<\ell_1-1$ and 
$$H=K_{\frac{\ell_1-2}{2}} \vee
\overline{K}_{\frac{\ell_1}{2}+s(\ell_1-1)+r}.$$
For the former one, we have $G=G_2(n,1,\ell_1)$ since $K_1\vee \EX (n-1, P_{\ell_1})=G_2(n,1,\ell_1)$ is $F$-free as shown before. For the latter one, We first prove the following claim.

\noindent \textbf{Claim 1.} If $\ell_1$ is even and $r \in \{\frac{\ell_1}{2},\frac{\ell_1-2}{2} \}$, then for $0\leq s\leq d-2$, $G(s)$ is $F$-free if and only if $s=0$, $a_q=\ell_1=4$ and $r=1$.
\begin{proof}
	Assume that $0\leq s\leq d-2$. 
	We first show that  $G(s)$ is $F$-free if and only if $(s+1)(\ell_1-1)+r\leq a_q$.  It is clear that $|H|=(s+1)(\ell_1-1)+r$ and $H$ has no $P_{\ell_1}$. If $(s+1)(\ell_1-1)+r\leq a_q$, then $H$ contains no $S_{a_j}$ and so $G(s)$ is $F$-free because 
	each $S_{a_j}$ and the $P_{\ell_1}$ in $F$ must use at least one vertex of the $K_q$,  which is impossible. If $(s+1)(\ell_1-1)+r\geq a_q+1$, then $H$ contains an $S_{a_q}$. Since $d-s-1\geq 1$,  take one copy of $K_{\ell_1-1}$ and one vertex of the $K_q$, we can get a $P_{\ell_1}$. If $q=1$, then $G(s)$ contains an $F$. If $q\geq 2$, delete the $S_{a_q}$ and $P_{\ell_1}$ from $G(s)$, we can see the remaining graph contains a subgraph $G_1(n-\ell_1-a_q-1,q-1)$. Because $n$ is sufficiently large,  $G_1(n-\ell_1-a_q-1,q-1)$ contains $\cup_{j=1}^{q-1}S_{a_j}$ as a subgraph, and so 
	$G(s)$ contains an $F$.
	
	Since $a_q\geq \ell_1-1$, and $\beta_F> \max\limits _{1 \leq j\leq q} \big\{j-1+\frac{a_j-1}{2}\big\}$ implies $a_q\leq \ell_1$, we have $a_q=\ell_1-1$ or $\ell_1$.
	By $(s+1)(\ell_1-1)+r\leq a_q$, we have $s\leq \frac{a_q-r}{\ell_1-1}-1$.
	Moreover,  note that $a_q\geq 3$, $s\geq 0$, $\ell_1$ is even and $r \in \{\frac{\ell_1}{2},\frac{\ell_1-2}{2} \}$,  we have $s=0$, $a_q=\ell_1=4$ and $r=1$.
\end{proof}

Since $G(d-1)=G_1(n,\ell_1/2)$ is $F$-free, by Claim 1, we have 
$$\mathbb{G}=\begin{cases}
	\{G(d-1), G(0) \}, & \mbox{if } a_{q}=\ell_1=4,\ r=1, \\[1pt]
	\{G(d-1)\}, & \mbox{otherwise.}
\end{cases}$$

Therefore, the extremal graphs in Part I of Theorem \ref{main} hold for $q=1$. Assume that $q\geq 2$. Then the equality holds in (\ref{q>1}) if and only if $G=K_1 \vee \EX (n-1, F-S_{a_1})$ is $F$-free. By induction hypothesis, we get the extremal graphs $\EX (n-1, F-S_{a_1})$ for $q-1$. And it is easy to verify that the extremal graphs in Part I of Theorem \ref{main} hold for $q$. 

The proof of Part I is completed.

\vskip 3mm
\noindent \textbf{Part II.} $\beta_F\leq \max\limits _{1 \leq j\leq q} \big\{j-1+\frac{a_j-1}{2}\big\}$.
\vskip 2mm

If $p=0$, then the result holds by Theorem \ref{Sncup}. So we assume $p\geq 1$.
Let 
$$H(n,j)=K_{j-1}\vee \EX\left(n-j+1, S_{a_j}\right),~~f_j=j-1+\frac{a_j-1}{2}$$ 
and $i^*$ be the index maximizing
the number of edges of $H(n,j)$, i.e., 
$$e(H(n,i^*))=\max\limits_{1\leq j \leq q} \{ e(H(n,j))\}=ex\left(n,\mathop{\cup}\limits_{j=1}^{q} S_{a _j}\right).$$  
Since $n$ is sufficiently large,  $f_{i^*}= \max\limits _{1 \leq j\leq q} f_j$. 

Before starting to prove, we need the following claims. 
\vskip 2mm
\noindent \textbf{Claim 2.} (\cite{lidicky2013})  For any $j<i^*$, $f_j < f_{i^*}$, i.e., $e(H(n,j))<e(H(n,i^*))$.

\vskip 2mm
\noindent \textbf{Claim 3.} (\cite{lidicky2013}) Let $F^*=S_{a_{i^*}} \cup S_{a_{i^*+1}} \cup \cdots \cup S_{a_{q}} $. Then $ex(n, F^*)= ex \left(n, S_{a_{i^*}}\right).$

\vskip 2mm
\noindent \textbf{Claim 4.} When $\beta_F\leq \max\limits _{1 \leq j\leq q} \{j-1+\frac{a_j-1}{2}\}$, $p\geq 1$ and $q\geq 1$, we have
\begin{equation*}\label{above}
	\binom{q}{2}+q(n-q)+ex\left(n-q, \mathop{\cup}\limits _{i=1}^{p} P_{\ell _i}\right) < ex \left(n,\mathop{\cup}\limits _{j=1}^{q} S_{a _j}\right).
\end{equation*}
\begin{proof}
	When $q=1$, since $p\geq 1$, by (\ref{case1}) and (\ref{case11}), we have $n-1+ ex(n-1, \mathop{\cup}\limits _{i=1}^{p} P_{\ell _i}) \leq \beta_F \cdot n -1<\lfloor \frac{a_1-1}{2} n \rfloor$.	

When $q>1$, by Theorem \ref{Sncup}, we have
\begin{equation}\label{star}
	ex \left(n,\mathop{\cup}\limits _{j=1}^{q} S_{a _j}\right)=e(H(n,i^*))=\left\lfloor f_{i^*}\cdot  n -\frac{(i^*-1)(i^*+a_{i^*}-1)}{2}  \right\rfloor.
\end{equation}	
If $\beta_F <  f_{i^*}$, then since $n$ is sufficiently large, the results follows. So we assume $\beta_F = f_{i^*}$. By (\ref{case1}), (\ref{case11}) and (\ref{star}), we only need to compare the terms that do not contain $n$.  
	
	If $p=1$, then $a_{i^*}=\ell_1+1+2q-2i^*$, and thus
	\begin{align*}
    	  \frac{(i^*-1)(i^*+a_{i^*}-1)+1}{2}& = \frac{(i^*-1)(\ell_1+2q-i^*)+1}{2}\\
    	  &\leq \frac{(q-1)(\ell_1+q)+1}{2}\\
    	  &< \frac{q^2+(\ell_1-1)(q+r)-r^2}{2}.
	\end{align*}
	
 Similarly, if $p\geq 2,  \ell_1=\cdots=\ell_{p}=3$, then $a_{i^*}=2(p+q+1-i^*)$, and thus
	\begin{align*}
		\frac{(i^*-1)(i^*+a_{i^*}-1)+1}{2}& = \frac{(i^*-1)(2p+2q+1-i^*)+1}{2}\\
		&\leq \frac{(q-1)(2p+q+1)+1}{2}\\
		&< \frac{(p+q)^2-1}{2}.
	\end{align*}

If $p\geq 2, \exists \ \ell_i\neq 3$, then $a_{i^*}=2q+2\delta_F +1-2i^*$, and thus
	\begin{align*}
		\frac{(i^*-1)(i^*+a_{i^*}-1)+1}{2}& = \frac{(i^*-1)(2q+2\delta_F -i^*)+1}{2}\\
		&\leq \frac{(q-1)(q+2\delta_F)+1}{2}\\
		&< \frac{1}{2} (q+\delta_F) (q+\delta_F -1)-1.
	\end{align*}
Therefore, Claim 4 holds by the above arguments. 
\end{proof}

Now, we begin to prove the result of this part.

Let $G$ be an extremal graph for $F$. Noting that any $\cup _{j=1}^{q} S_{a _j}$-free graph is also  $\mathop{\cup}\limits _{i=1}^{p} P_{\ell _i} \mathop{\cup}\limits _{j=1}^{q} S_{a _j}$-free, we have $e(G)\geq ex(n, \cup _{j=1}^{q} S_{a _j})$. We will show $e(G)\leq ex(n, \cup _{j=1}^{q} S_{a _j})$ by induction on $q$.

Assume that $q=1$. If $G$ contains $S_{a_1}$ as a subgraph, then $G-S_{a_1}$ is $\cup _{i=1}^{p} P_{\ell _i}$-free and $ex\left(n, \mathop{\cup}\limits _{i=1}^{p} P_{\ell _i}\right) \leq (\beta_F-1)n$ by (\ref{case1}). Since $H(n,1)$ is $\cup _{i=1}^{p} P_{\ell _i} \cup S_{a_1}$-free, we have $e(G)\geq e(H(n,1))=ex(n, S_{a_1})$. Let $m_0$ be the number of edges incident with the vertices of $S_{a_{1}}$ in $G$. Then by $\beta_F\leq f_1$ and (\ref{star}), we have
\begin{equation*}\label{eqs3}
 m_0  =e(G)-e(G-S_{a_1})\geq e(H(n,1))- ex\left(n-a_1-1, \mathop{\cup}\limits _{i=1}^{p} P_{\ell _i}\right) \geq n-1. 
\end{equation*}
Since $n$ is sufficiently large, the $S_{a_1}$ has a vertex $v$ such that $d(v)\geq \sum_{i=1}^{p} \ell_i+a_1$. It implies that $G-v$ is $\cup _{i=1}^{p} P_{\ell _i}$-free and so $e(G-v)\leq 
ex(n-1, \cup _{i=1}^{p} P_{\ell _i})$. Because $e(G)=d(v)+e(G-v)$ and $d(v)\leq n-1$, we have $e(G)<  ex(n, S_{a_1})$ by Claim 4,  which contradicts $e(G)\geq ex(n, S_{a_1})$. Therefore, $G$ is $S_{a_1}$-free and $e(G)\leq ex(n, S_{a_1})$, with equality if and only if $G=\EX (n, S_{a_1} )$.

 Suppose that $q> 1$ and Theorem \ref{main} holds for any positive integer smaller than $q$. We complete the proof by considering the following two cases separately.

\vskip 2mm 
 \noindent\textbf{Case 1.} $1<i^*\leq q$.
\vskip 2mm

By Claim 2, $ex(n,\cup _{j=1}^{i^*-1} S_{a _j})= \max\limits_{1\leq j \leq i^*-1} e(H(n,j)) < e(H(n,i^*))=ex(n,\cup _{j=1}^{q} S_{a _j})\leq e(G)$, which implies $G$ contains $\cup _{j=1}^{i^*-1} S_{a _j}$ as a subgraph.

For any $i<i^*$, let $j'$ be the index maximizing the number of edges for $\cup _{j=1}^{q} S_{a _j}-S_{a_i}$, where the index of $S_{a_j}$ is $j-1$ for $i+1\leq j\leq q$.  Clearly, the index of $S_{a_{i^*}}$  is $i^*-1$ in $\cup _{j=1}^{q} S_{a _j}-S_{a_i}$. Let $F'=F-S_{a_i}$.
Since $\beta_{F'}=\beta_F-1 \leq f_{i^*}-1=(i^*-1)-1+\frac{a_{i^*}-1}{2}$, by induction hypothesis,  $ex(n, F')= ex(n,\cup _{j=1}^{q} S_{a _j}-S_{a_i})$. 

Let $m_0$ be the number of edges incident with the vertices of $S_{a_{i}}$ in $G$. 
If $j'<i$, then $f_{j'}<f_{i^*}$ by Claim 2,  and hence we have
\begin{align*}
	m_0 & =e(G)-e(G-S_{a_i})\geq e(H(n,i^*))- ex(n, F')\geq  f_{i^*} \cdot n -f_{j'}\cdot  n +O(1)=\Omega (n).
\end{align*}
If $j'\geq i$, then $j'=i^*-1$ and $(i^*-1)-1+\frac{a_{i^*}-1}{2}= f_{i^*}-1$, and so  $e(G-S_{a_i}) \leq ex(n, F')\leq (f_{i^*}-1 ) \cdot n$. By a similar calculation as above, we get that $m_0 = \Omega (n) $.

Because $m_0 = \Omega (n) $ and $n$ is sufficiently large, there is a vertex $v$ in $S_{a_i}$ such that $d(v)\geq \sum_{i=1}^{p} \ell_i+\sum_{j=1}^{q} a_j +q-1$  for any $i<i^*$. Therefore, there is a set $U$ of order $i^*-1$ in $G$ such that each vertex in $U$ has degree of at least $\sum_{i=1}^{p} \ell_i+\sum_{j=1}^{q} a_j +q-1$. 

Let $\widetilde{F}=F-\cup_{j=1}^{i^*-1} S_{a_j} $. Now we show that $G-U$ is $\widetilde{F}$-free. If $G-U$ contains $\widetilde{F}$ as a subgraph, we set $W=V(G-U)\setminus V(\widetilde{F} )$. Then $|W| \geq \sum_{j=1}^{i^*-1} a_j$ because $n$ is sufficiently large. For any $u\in U$, we have 
$$d_{G[W]}(u) \geq  \sum\limits_{i=1}^{p} \ell_i+ \sum\limits_{j=1}^{q} a_j +q-1 -(q-1) -\sum_{i=1}^{p} \ell_i -\sum_{j=i^*}^q a_j = \sum_{j=1}^{i^*-1} a_j.$$
This means that $G-\widetilde{F}$ contains $\cup_{j=1}^{i^*-1} S_{a _j}$ as a subgraph with $i^*-1$ center vertices in $U$ and $\sum_{j=1}^{i^*-1} a_j$ pendant vertices in $W$, and hence we have $F \subseteq G$,  a contradiction. Therefore, $G-U$ is $\widetilde{F}$-free.

Since $\beta_{\widetilde{F}}=\beta_F-i^*+1 \leq f_{i^*}-i^*+1$, by induction hypothesis, $ex(n, \widetilde{F})=ex(n, F^*)$. By Claim 3, we have 
\begin{align*}
	e(G)&\leq e(G[U])+e(U,V(G)\setminus U)+e(G-U)\\
	&\leq \binom{i^*-1}{2}+(i^*-1)(n-i^*+1)+ex(n-i^*+1,S_{a_{i^*}})=ex\left(n, \mathop{\cup}\limits _{j=1}^{q} S_{a _j}\right).
\end{align*}
 Equality holds if and only if $G=\EX (n, \cup _{j=1}^{q} S_{a _j} )$.

\vskip 2mm
 \noindent\textbf{Case 2.} $i^*=1$.
 \vskip 2mm

 Then $e(G)\geq e(H(n,1))=ex(n, S_{a_1})=ex(n, \cup _{j=1}^{q} S_{a _j})$. Assume that $G$ contains an $S_{a_1}$. Let $F''=F-S_{a_1}$. Then $G-S_{a_1}$ is $F'' $-free. Let $j''$ be the index maximizing the number of edges for $\cup _{j=2}^{q} S_{a _j}$, where the index of $S_{a_j}$ is $j-1$ for $2\leq j\leq q$. Clearly, 
 $j''-1+\frac{a_{j''+1}-1}{2}=f_{j''+1}-1\leq f_1-1 = \frac{a_1-1}{2}-1$.

If $\beta_{F''}\leq f_{j''+1}-1$, then by induction hypothesis, $ex(n,F'')= ex(n, \cup _{j=2}^{q} S_{a _j})$, and thus
\begin{equation}\label{a1}
	\begin{split}
		n-1+ ex(n-1, F'')& \leq n-1+ ex(n, F'')\leq n-1+ (f_{j''+1}-1) n \\
		& \leq  \frac{a_1-1}{2}n-1 < ex(n, S_{a_1}).
	\end{split}
\end{equation}

By applying the similar discussion as in the proof of Part I, we can show that (\ref{q>1}) holds for any $F$ satisfying $\beta_F> \max\limits _{1 \leq j\leq q} \big\{j-1+\frac{a_j-1}{2}\big\}$. If $\beta_{F''}=\beta_F-1 > f_{j''+1}-1$, then by (\ref{q>1}) and (\ref{case1}), we have $ex(n,F'')\leq  (\beta_F-1 ) \cdot n$. Moreover, by (\ref{q>1}) and Claim 4, we have
\begin{equation}\label{a2}
	\begin{split}
		&n-1+ ex(n-1, F'')\\
		 \leq &\binom{q}{2}+q(n-q)+ex\left(n-q, \mathop{\cup}\limits _{i=1}^{p} P_{\ell _i}\right)< ex \left(n,\mathop{\cup}\limits _{j=1}^{q} S_{a _j}\right).
	\end{split}
\end{equation}

Let $m_0$ be the number of edges incident with the vertices of $S_{a_{1}}$ in $G$. Then by (\ref{star}), we have
\begin{align*}
	m_0 & =e(G)-e(G-S_{a_1})\\
	&\geq e(H(n,1))- ex\left(n-a_1-1, F''\right) \\
	&\geq f_1 \cdot n -\max\{f_1-1, \beta_F-1  \}\cdot n -1 \geq n-1,
\end{align*}
thus there is a vertex $v$ with degree  of at least $\sum_{i=1}^{p} \ell_i+\sum_{j=1}^{q} a_j +q-1$ in $S_{a_1}$. It implies that $G-v$ is $F''$-free. Hence, $e(G)=d(v)+e(G-v)\leq n-1+ ex(n-1, F'') < ex(n, S_{a_1})$ by (\ref{a1}) and (\ref{a2}), which contradicts $e(G)\geq ex(n, S_{a_1})$. Therefore, $G$ is $S_{a_1}$-free and $e(G)\leq ex(n, S_{a_1})$, with equality if and only if $G=\EX (n, S_{a_1} )$. \hfill $\square $

\section{The Tur\'an number of $\mathop{\cup}\limits _{i=1}^{p} P_{\ell _i} \cup q S_a$}\label{remark}

\hspace{1.5em}In this section, we further improve the lower bound of $n$ obtained in Theorem \ref{main} for $F$ in which all stars  have the same order. Let $F=\cup _{i=1}^{p} P_{\ell _i} \cup q S_a$ and set
 $$N_1=(a^2+a+1)q+ \begin{cases}
	(\frac{\ell_1}{2}+1)a+\ell_1, & \text{if } p=1,\\
	(2a+3) p+\frac{a+1}{2}, & \text{if } p\geq 2, \ell_i= 3 \text{ for } i\in [p],\\
	 \max \left\{a(s-\delta_F+1)+s, L  \right\}, & \text{if } p\geq 2, \ell_i\neq 3 \text{ for } i\in [p],
\end{cases}$$
 $$N_2=(a^2+a+1)q+ \begin{cases}
	(a+1)s, & \text{if } p=1, \text{ or } p\geq 2, \ell_i= 3 \text{ for } i\in [p],\\
    \max \left\{(a+1)s, L  \right\}, &  \text{if } p\geq 2, \ell_i\neq 3 \text{ for } i\in [p],
\end{cases}$$
where $s=\sum_{i=1}^{p} \ell_i$ and $L=\begin{cases} 
	\frac{\delta_F}{2}( 1+\frac{\ell_1-2}{2\delta_F-\ell_1}) +1, & \text{if all } \ell_i \text{ is even,}\\[2pt]
	\frac{5(s-1)(s-2)^2}{4(\delta_F-1)}+\delta_F-1, &  \text{otherwise.}
\end{cases}$
\vskip 3mm
The main result of this section is the following.
\begin{theo}\label{cupPl}
	Let $F=\cup _{i=1}^{p} P_{\ell _i} \cup q S_a$ with $a\geq 3$,  and $\ell_i\neq 3$ for each $i\in [p]$ or $\ell_1=\cdots=\ell_ p=3$ when $p\geq 2$. Then 
   $$ex(n, F) =\binom{q}{2}+q(n-q)+ex(n-q, \cup _{i=1}^{p} P_{\ell _i})$$
  if $n\geq N_1 $, $a\leq 2 \delta_F-2\mu_F+2$ and $a\geq \ell_1-1$ in the case when $p=1$, and $ex(n, F) =ex(n, qS_a)$ if $n\geq N_2 $ and $a\geq 2 \delta_F-2\mu_F+3$. 
  
  Moreover, in the former case, $\EX(n,F)=K_{q}  \vee \EX(n-q, \cup _{i=1}^{p} P_{\ell _i})$ if $p\geq 2$ and
  $$\E(n,F)=\begin{cases}
  	\big\{G_2\big(n,q,\ell_1\big) \big\}\cup \mathcal{G},	& \text{ if } \ell_1 \text{ is even and } r \in \big\{\frac{\ell_1}{2},\frac{\ell_1-2}{2} \big\}, \\[2pt]
  	\big\{G_2\big(n,q,\ell_1\big)\big \}, & \text{ otherwise, }  
  \end{cases}$$
  when $p=1$ and  $a\geq \ell_1-1$, where $\,n=q+d(\ell_1 -1)+r$, $0\leq r < \ell_1-1$, 
  $$\mathcal{G}= \begin{cases}
  	\left\{ G(d-1),  G(0) \right\}, & \text{if } a=\ell_1=4, r=1, \\
  	\left\{G(d-1) \right\}, & \text{otherwise},
  \end{cases}$$
  and $G(s)= K_{q} \vee \left( (d-s-1) K_{\ell_1-1}\cup 
  \left(K_{\frac{\ell_1-2}{2}} \vee
  \overline{K}_{\frac{\ell_1}{2}+s(\ell_1-1)+r}
  \right) \right)$; in the latter case, $\EX(n,F)=\EX(n,qS_a)$.
\end{theo}

As a special case, take  $\ell_i= \ell$ for $i\in [p]$ and $a=\ell-1$ in Theorem \ref{cupPl}, we can obtain Theorems \ref{1,k} and  \ref{k1k2}.
\vskip 3mm

Since the proof of Theorem \ref{cupPl} is almost the same as that of Theorem \ref{main} except some additional computation, we only give a sketch of the proof. The following lemma is used to ensure that the extremal graph for $F=\cup _{i=1}^p P_{\ell _i}$ with even integers $\ell_i$ is unique when $n$ is larger than some constant. 

\begin{lem}\label{UP2l}
	Let $F=\cup _{i=1}^p P_{\ell _i}$, $p\geq 2$, $\ell_1\geq \ell_2 \geq \cdots \geq \ell_p$ be positive even integers, and $n> \max \left\{3\delta_F-1,\frac{\delta_F}{2}\left( 1+\frac{\ell_1-2}{2\delta_F-\ell_1}\right) \right\}  $. Then 
	$ex(n,F)=  \left(\delta_F-1\right)\left(n-\frac{\delta_F}{2}\right)$, and \EX$(n,F)=G_1(n, \delta_F-1)$.
\end{lem}
\begin{proof}
	We will show that 
	$\left[n,\sum_{i=1}^j \ell_i, \ell_j\right]<(\delta_F-1)(n-\frac{\delta_F}{2})$ for any $j\in[p]$, and then the result follows from Theorem \ref{P2n}. 
	
	Since $[n,m,\ell] \leq \binom{m-1}{2} +\frac{n-(m-1)}{\ell-1} \binom{\ell-1}{2} =\frac{1}{2} (n(\ell-2)+(m-1)(m-\ell))$, we have
	\begin{align*}
		\left[n,\sum\limits_{i=1}^j \ell_i, \ell_j\right] & \leq \frac{1}{2} \left(n(\ell_j-2)+\left(\sum\limits_{i=1}^j \ell_i-1\right)\left(\sum\limits_{i=1}^j \ell_i-\ell_j\right) \right)\\
		&\leq \begin{cases}
			\frac{1}{2} ( n(\ell_j-2)+(2\delta_F-1)(2\delta_F-\ell_j) ), & j\geq 2,\\
			\frac{1}{2} n(\ell_1-2), & j=1.
		\end{cases}
	\end{align*}

	When $j\geq 2$, by $n> 3\delta_F-1$, we have $\left(\delta_F-1\right)\left(n-\frac{\delta_F}{2}\right)> \left[n,\sum_{i=1}^j \ell_i, \ell_j\right]$.
	
	When $j=1$, by $n> \frac{\delta_F}{2}\left( 1+\frac{\ell_1-2}{2\delta_F-\ell_1}\right)$, we have $\left(\delta_F-1\right)\left(n-\frac{\delta_F}{2}\right)>\left [n,\ell_1, \ell_1\right]$.
	
	The proof is completed.
\end{proof}

\noindent{\it The sketch of the proof of Theorem \ref{cupPl}.}  To complete the proof, we need only  to show the statements hold for sufficiently large $n$ in Theorem \ref{main} still hold for $n\geq N_1$ in Part I and $n\geq N_2 $ in Part II.

Let $F'=\cup _{i=1}^{p} P_{\ell _i}$, where $\ell_i\neq 3$ for $i\in [p]$ or $\ell_1=\cdots=\ell_ p=3$ when $p\geq 2$.

In Part I ($3\leq a\leq 2 \delta_F-2\mu_F+2$ and $a\geq \ell_1-1$ at this time), we need to show that $n\geq N_1$ can ensure the following.

\noindent(1) Theorems \ref{p1}, \ref{pp} and \ref{p3} still hold for $F'$, i.e., the Tur\'an number and corresponding extremal graphs are unchanged;\\
(2) $e(G)> ex(n, q S_a)$ in (\ref{exSn}), and thus $G$ contains $qS_a$ as a subgraph;\\
(3) $m_0\geq (a+1)\left( \sum_{i=1}^{p} \ell_i+ qa+q -1\right)$ in  (\ref{m0}).

To ensure (1), it is sufficient to show Theorem \ref{pp} holds for $F'$ when $n\geq N_1$, which follows from Theorem \ref{kPnZ} and Lemma \ref{UP2l}. 
By calculation and Theorem \ref{lanS}, (2) holds, and (3) follows from an easy calculation.

In Part II ($a\geq 2 \delta_F-2\mu_F+3$ at this time),  we always have $i^*=q$ and so we only need to show that $n\geq N_2 $ can ensure the following.

\noindent (a) Theorems \ref{p1}, \ref{pp} and \ref{p3} still hold for $F'$, i.e., the Tur\'an number are unchanged;\\
(b) $m_0\geq (a+1)\big(\sum_{i=1}^{p} \ell_i+a\big)$ when $q=1$;\\
(c) $m_0\geq (a+1)\big(\sum_{i=1}^{p} \ell_i+qa +q-1\big)$ in Case 1.

Noting the range of $n$, we can check that (a) follows from Theorem \ref{kPnZ} and Lemma \ref{UP2l}, and (b) and (c) hold by an easy calculation.

Thus Theorem \ref{cupPl} holds.

\vskip 2.5em

\section*{Acknowledgements}
We are grateful to the anonymous referees for their very careful comments which help us improve the presentation of 
this paper. This research was supported by NSFC under grant numbers 12371347, 12271337 and 12471327.

\vskip 1.5em 

\noindent\textbf{Data availability statement} \  Data sharing not applicable to this article as no datasets were generated or analysed during the current study.

\section*{Declarations}

\textbf{Conflict of interest}\  The authors declare that they have no known competing financial interests or personal relationships that could have appeared to influence the work reported in this paper.

\end{document}